\documentclass[11pt]{article}
\usepackage{graphicx}
\usepackage{tikz}
\usepackage{amssymb}
\usepackage{amsmath,amsthm}
\usepackage{verbatim}
\usepackage{amsfonts}
\usepackage[T1]{fontenc}
\usepackage{multicol}
\usepackage{color}

\newtheorem{thm}{Theorem}
\newtheorem{prop}{Proposition}[section]

\newtheorem{lem}[thm]{Lemma}

\newtheorem{conj}[thm]{Conjecture}

\newtheorem{claim}[prop]{Claim}

\newcommand{\A}{\mathcal{A}}

\newcommand{\C}{\mathcal{C}}

\newcommand{\E}{\mathcal{E}}

\newcommand{\G}{\mathcal{G}}
\newcommand{\Hy}{\mathcal{H}}
\newcommand{\I}{\mathcal{I}}

\newcommand{\Sym}{\mathcal{S}}

\newcommand{\bP}{\mathbb{P}}
\newcommand{\bE}{\mathbb{E}}
\newcommand{\bI}{\mathbb{I}}

\newcommand{\bN}{\mathbb{N}}
\newcommand{\bfa}{\mathbf{a}}
\newcommand{\bfb}{\mathbf{b}}

\newcommand{\bfx}{\mathbf{x}}
\newcommand{\bfy}{\mathbf{y}}

\newcommand{\bfzero}{\mathbf{0}}
\newcommand{\Supp}{\textrm{Supp}}
\newcommand{\Var}{\textrm{Var}}
\newcommand{\Cov}{\textrm{Cov}}

\newcommand{\X}{\mathcal{X}}
\newcommand{\mfX}{\mathfrak{X}}
\newcommand{\ep}{\varepsilon}
\newcommand{\Z}{\mathbb{Z}}

\title{Zero sums in restricted sequences}
\author{Niranjan Balachandran\footnote{Dept.of Mathematics, 
IIT Bombay, Mumbai. email: niranj (at) math.iitb.ac.in. 
Supported by grant 12IRCCSG016, IRCC, IIT Bombay}\ \ 
and Eshita Mazumdar\footnote{Center for Combinatorics, 
Nankai University, Tianjin, China. email: eshitamazumdar@yahoo.com.
Supported by NSFC with grant no. 11681217} }

\begin{document}
\maketitle
\begin{abstract} Suppose $A\subset \Z_n\setminus\{0\}$. 
A sequence $\bfx=(x_1,\ldots,x_m)$ of elements of $\Z_n$ is called an  \textit{$A$-weighted Davenport Z-sequence} if there 
exists $\bfa:=(a_1,\ldots,a_m)\in (A\cup\{0\})^m\setminus\bfzero_m$ such that $\sum_i a_ix_i=0$, where $\bfzero_m=(0,\ldots,0)\in\Z_n^m$. Similarly, the sequence 
$\bfx$ is called an \textit{$A$-weighted Erd\H{o}s Z-sequence} if there exists $\bfa=(a_1,\ldots,a_m)\in (A\cup\{0\})^m$ with $|\Supp(\bfa)|=n$, such that $\sum_i a_ix_i=0$, where $\Supp(\bfa):=\{i: a_i\ne 0\}$.  A $\Z_n$-sequence $\bfx$ is called  $k$-restricted if no element of $\Z_n$ appears more than $k$ times in $\bfx$.
In this paper, we study the problem of determining the least value of $m$ for which a $k$-restricted $\Z_n$-sequence of length $m$ is an $A$-weighted Davenport Z-sequence (resp. an $A$-weighted Erd\H{o}s Z-sequence). We also consider the same problem for random $\Z_n$-sequences and some very natural choices of the set $A$.
 \end{abstract}


2010 AMS Classification Code:  11B50, 11B75, 11P70, 11K99.

\section{Introduction}
 In this paper, $k,\ell,m,n$ shall always refer to positive integers. By $[n]$ we shall mean the set $\{1,\ldots,n\}$, and for integers $a<b$, $[a,b]$ shall denote the set $\{a,a+1,\ldots,b\}$. By $\Z_n$ we shall denote the cyclic group of order $n$.  
 
 Throughout this paper, we shall use the Landau asymptotic notation: For functions $f,g$, we write $f(n)=O(g(n))$ if there exists an absolute constant $C>0$ and an integer $n_0$ such that for all $n\ge n_0, |f(n)|\le C|g(n)|$. We write $f=\Omega(g)$ if $g=O(f)$, and we write $f=\Theta(g)$ if $f=O(g)$ and $f=\Omega(g)$. We also write $f=o(g)$ if $\displaystyle\lim_{n\to\infty}\frac{f(n)}{g(n)}=0$. For a real $a>1$, we write $\log_a n$ to denote the logarithm of $n$ to the base $a$.

 For a finite abelian group $(G, +)$ a {\it $G$-sequence of length $m$} shall refer to a sequence $\bfx:=(x_1,\ldots,x_m)$ with $x_i\in G$ for all $i$. If some $x\in G$  appears $r$ times in $\bfx$ then we say that $x$ has  {\it multiplicity} $r$ in $\bfx$. In particular, if $x$ does not appear in $\bfx$ then $x$ has multiplicity zero in $\bfx$. For a sequence $\bfx=(x_1,\ldots,x_m)$, and for a subset 
 $I\subseteq [m]$ of the set of indices, we shall denote by  $\bfx_I$ the sum $\displaystyle\sum_{i\in I} x_i$.  If $I=\emptyset$ then this corresponds to the empty sum.  For  sequences $\bfx=(x_1,\ldots,x_m), \bfy=(y_1,\ldots,y_m)$ of  the same length, we shall denote by $\langle \bfx,\bfy\rangle$ the sum $\displaystyle\sum_{i\in [m]} x_iy_i$.  If $\bfa=(a_1,\ldots,a_m)$ and  $\bfx=(x_1,\ldots,x_m)$ are $\Z_n$-sequences, then  $\bfa\cdot\bfx$ shall denote the sequence $(a_1x_1,\ldots, a_mx_m)$, where the multiplication is from the ring structure of $\Z_n$. For $\bfa\in\Z_n^m$, we define $\Supp(\bfa):=\{i: a_i\ne 0\}$. Finally, by $\bfzero_m$ we shall mean the zero sequence $(0,\ldots,0)\in(\Z_n)^m$.
 
For $A\subseteq\Z_n\setminus\{0\}$, a sequence $\bfx=(x_1,\ldots,x_m)$ of elements of $\Z_n$ is called an  \textit{$A$-weighted Davenport Z-sequence} if there exists $\bfa=(a_1,\ldots,a_m)\in (A\cup\{0\})^m\setminus\bfzero_m$ such that $\langle\bfa,\bfx\rangle=0$. In words, there is a choice of `coefficients' from $A$, not all zero, such that the corresponding `linear combination' equals zero. Similarly, the sequence $\bfx$ is called an \textit{$A$-weighted Erd\H{o}s Z-sequence} if there exists $\bfa:=(a_1,\ldots,a_m)\in (A\cup\{0\})^m$ with $|\Supp(\bfa)|=n$, such that $\sum_i a_ix_i=0$. When $A=\{a\}$ for some $a$ co-prime to $n$, we shall refer to such a sequence simply as a Davenport Z-sequence (resp. an Erd\H{o}s Z-sequence). When the set $A$ is clear from the context, we shall drop any mention of the set $A$ and refer simply to  weighted Davenport Z-sequences (resp. weighted Erd\H{o}s Z-sequences). 
 
 The notion of a weighted Davenport Z-sequence draws its motivation from a well-studied combinatorial invariant associated with a finite abelian group $G$, namely the {\it Davenport constant of $G$} (denoted $D(G)$), which is defined as the least positive integer $m$ such that every $G$-sequence of length $m$ admits a non-trivial finite subsequence whose sum is zero in $G$.  A generalization of this notion in \cite{Adhetal, AC} \footnote{While the generalization was for all finite abelian groups, we shall restrict our attention to the cyclic group $\Z_n$.}   introduces a weighted version of this combinatorial invariant as follows: For a given $A\subseteq\Z_n\setminus\{0\}$,  {\it the weighted Davenport constant of $\Z_n$ with respect to weight set $A$} (denoted by  $D_A(\Z_n)$) is the least integer $m$ such that for every $\Z_n$-sequence $\bfx=(x_1,\ldots,x_m)$ of length $m$, there exists $\bfa\in (A\cup\{0\})^m\setminus\bfzero_m$ satisfying  $\langle\bfa,\bfx\rangle=0$.  
  
 The following results are well known:
 \begin{itemize}
 \item If $a\in\Z_n^*$ and $A=\{a\}$ then any sequence of length $n$ is an $A$-weighted Davenport Z-sequence. The bound $n$ is sharp, as is witnessed by the sequence $(\underbrace{1,\ldots,1}_{n-1 \text{ times}})$. This is folklore and is a simple exercise on the application of the Pigeonhole Principle.
 \item  (See \cite{Adhetal}) For the set $A=\{1,-1\}$, any $\Z_n$-sequence of length 
 $\lfloor\log_2 n\rfloor +1$ is a weighted Davenport Z-sequence, 
 and this result is again, best possible: The sequence 
 $(1,2,\ldots,2^{k-1})$, for $k=\lfloor\log_2 n\rfloor$, is not a weighted Davenport Z-sequence. 
 \item (See \cite{Gri}, \cite{Luca})  If $n = q_1\cdots q_a$ is the product of $a$ primes (not necessarily distinct) and $A=\Z_n^*$, the group of units of the ring $\Z_n$  then any sequence of length $a+1$ is an $A$-weighted Davenport Z-sequence, and again, this result is best possible: The sequence $(1, q_1, q_1q_2,\ldots, q_1q_2\cdots q_{a-1})$ is not a weighted Davenport Z-sequence.
 \end{itemize}
 
 The notion of an Erd\H{o}s Z-sequence draws its motivation from the following non-trivial theorem of Erd\H{o}s-Ginzburg-Ziv \cite{EGZ}: Every $\Z_n$-sequence of length $2n-1$ 
 admits a subsequence of size $n$ whose sum equals  zero. Thus, in our language,  every $\Z_n$-sequence of length $2n-1$ is an Erd\H{o}s Z-sequence, and again, this is best possible since 
 the sequence $(\underbrace{0,\ldots,0}_{n-1 \text{ times}}, \underbrace{1,\ldots,1}_{n-1 \text{ times}})$ is not an Erd\H{o}s Z-sequence. For an arbitrary set $A\subseteq \Z_n\setminus\{0\}$, one can analogously define the parameter 
 $E_A(\Z_n)$  as the least integer $m$ such that for every $\Z_n$-sequence $\bfx=(x_1,\ldots,x_m)$, there exists  $\bfa=(a_1,\ldots,a_m)\in (A\cup\{0\})^m$ with $|\Supp(\bfa)|=n$ such that $\langle\bfa,\bfx\rangle=0$. It turns out  \cite{YuanZeng} that $E_A(\Z_n)=D_A(\Z_n)+n-1$. In particular, it follows that for $A=\{1,-1\}$, every sequence of length $n+\lfloor\log_2 n\rfloor$ is an $A$-weighted Erd\H{o}s Z-sequence and this result is best possible.
 
 One distinct feature of the aforementioned results pertaining to Erd\H{o}s Z-sequences (and some others) is that the exact values of the invariants $D_A(\Z_n)$ are witnessed by highly structured sequences. For instance, for $A=\{a\}$ (for any $a\in\Z_n^*$), the maximal sequences $\bfx$ that are not Davenport Z-sequences are necessarily of the form $\bfx=(\underbrace{x,\ldots,x}_{n-1 \text{ times}})$  for some $x\in\Z_n^*$. However, if we restrict our attention to $\Z_n$-sequences with a bound on the number of incidences of any particular element of $\Z_n\setminus\{0\}$, then it is conceivable that among this restricted class of sequences, the minimum value of $m$ for which every restricted $\Z_n$-sequence (restricted in this sense, which we shall make more precise soon) of length $m$ is an $A$-weighted Davenport (resp. Erd\H{o}s) Z-sequence, might be considerably smaller. And this is the focal point of this paper: How large must a {\it restricted} sequence be if it is to be a Davenport (resp. Erd\H{o}s) Z-sequence?  
 
To make precise what we mean by the word restricted, we define the following set of  $\Z_n$-sequences:
$$\X_k(n,m):=\left\{\bfx\in(\Z_n)^m: \textrm{every }x\in\Z_n\textrm{\ has multiplicity at most }k\textrm{\ in }\bfx\right\}.$$
We shall resort to some abuse of notation and denote $\X_k(n,m)$ by $\X_k(m)$ for simplicity.  The preceding discussion leads us to the following natural problem:
 
 \textbf{Problem:} \label{prob1} Let $A\subset\Z_n\setminus\{0\}$ and suppose $1\le k\le n-1$ is a positive integer.  Determine the least integer $m$ such that every $\bfx\in\X_k(m)$  is an $A$-weighted Davenport (Erd\H{o}s) Z-sequence.
 
 This problem is part of a larger umbrella of problems that usually go by the name of  Inverse Zero-Sum problems (see \cite{GGS}). Indeed, some results in \cite{GGS} are similar to some of the ones in this paper. However, the focus of the results is not the same as ours. 

We set up some further notation. For $k\in\bN$, by $\mathsf s^{{(k)}}(\Z_n)$, we shall mean the least integer $m$ such that every $\bfx\in\X_k(m)$ is an Erd\H{o}s Z-sequence. The case $k=1$ is usually awarded greater status and is referred to as the {\it Harborth constant} of $\Z_n$ (see \cite{MORS} for more results on the Harborth constant) and we shall not delve into that in this paper.

 Our first result of this paper addresses the aforementioned problem posed above for $k\ge 2$:

\begin{thm}\label{X_k(m)}
  Suppose $n-1\ge k\ge 2$ is an integer.  
  \begin{enumerate}
  \item  $\mathsf s^{{(2)}}(\Z_n)= n+ 2$ for $n\ge 5$.
  \item  For any prime $p\ge 5$, $\mathsf s^{(k)}(\Z_p)\leq p+k$. Furthermore, for each $k$, there exists an integer $p_0(k)$ such that
 $\mathsf s^{(k)}(\Z_p) = p+k \text { for all primes } p\geq p_0(k).$
  \item There exist constants $c,C>0$ such that every sequence $\bfx\in\X_k(C\sqrt{nk})$ is a Davenport Z-sequence. Furthermore, there exist sequences $\bfy\in\X_k(c\sqrt{nk})$ that are not Davenport Z-sequences.   \end{enumerate}\end{thm}


We now turn to the second objective of this paper, namely, the study of how these zero-sum related invariants behave if the $\Z_n$-sequence is picked {\it randomly}. More precisely, suppose $\mfX_m=(X_1,\ldots,X_m)$ is a $\Z_n$-sequence where each $X_i$ is picked independently and uniformly at random from $\Z_n$. One of the motivations to study random $\Z_n$-sequences from the zero-sum problem perspective is that relatively short random sequences are very likely to be members of $\X_k(m)$. Indeed, for any fixed $k\ge 2$, the probability that a random sequence $\mfX_m$ contains some element of $\Z_n$ with multiplicity at least $k+1$ is at most $\binom{m}{k+1}(1/n)^k\le\frac{m^{k+1}}{n^k}\to 0$ if $m=o(n^{k/(k+1)})$. Hence, if $m=o(n^{k/(k+1)})$ the sequence $\mfX_m$ is an element of $\X_k(m)$ \textit{with high probability (whp for short)}, i.e., with probability approaching $1$ as $n\to\infty$.  One also expects that with high probability, the random sequence $\mfX_m$ demonstrates no algebraic structure, so it is conceivable that random sequences of length significantly shorter than $D_A(\Z_n)$ are very likely to be weighted Davenport (resp. Erd\H{o}s) Z-sequences. This leads us naturally to the following problem:

\textbf{Problem:} Let $A\subset\Z_n\setminus\{0\}$ . Determine the least $m$ such that  a random $\Z_n$-sequence of length $m$ is a weighted Davenport (Erd\H{o}s) Z-sequence \textit{whp}. 
 
Our other results of this paper address this problem for weighted Erd\H{o}s Z-sequences, and weighted Davenport Z-sequences respectively for some natural choices of the set $A$:
  \begin{thm}\label{RandErd}
  Let $\mfX_m=(X_1,\ldots,X_{m})$ be a random $\Z_n$-sequence. Then \textit{whp} (as $n\to\infty$) the following statements hold:
  \begin{enumerate}
  \item $\mfX_{n+2}$ is an Erd\H{o}s Z-sequence.
  \item Let $A=\{a,b\}$. Suppose $A\cap\Z_n^*\ne\emptyset$ and suppose also that $a+b,a-b\in\Z_n^*$. Then $\mfX_n$ is an $A$-weighted Erd\H{o}s Z-sequence.
  \item  Let $A=\{1,-1\}$. Then $\mfX_{n+1}$ is an $A$-weighted Erd\H{o}s Z-sequence.
     \end{enumerate} 
  \end{thm}
    \begin{thm}\label{RandDav}
 Let $\mfX_m=(X_1,\ldots,X_{m})$ be a random $\Z_n$-sequence and suppose $\omega(n)$ is a function that satisfies $\omega(n)\rightarrow \infty $ as $n \rightarrow\infty$. 
 \begin{enumerate}
 \item The following hold  \textit{whp} (as $n\to\infty$) :
 \begin{eqnarray*}  \mfX_m\textrm{\ is\ a Davenport\ Z-sequence\ if\  }m&\ge& \log_2 n +\omega(n),\\
                                             \mfX_m\textrm{\ is\ not\ a\ Davenport\ Z-sequence\ if\  }   m&\le&  \log_2 n - \omega(n).\end{eqnarray*}
 \item Suppose $A=\{-1,1\}$. Then \textit{whp} (as $n\to\infty$) the following hold:
       \begin{eqnarray*} \mfX_m\textrm{\ is\ an }A\textrm{-weighted\ Davenport\ Z-sequence\ if\ } m&\ge& \log_3 n +\omega(n). \\
                                    \mfX_m\textrm{\ is\ not\ an }A\textrm{-weighted\ Davenport\ Z-sequence\ if\ }            m&\le&  \log_3 n - \omega(n).\end{eqnarray*}
\item Suppose $n=p_1\cdots p_r$ where $p_i$ are distinct odd primes and let $A=\Z_n^*$.  Then
\label{squarefree1} if $m\ge \omega(n)$ then $\mfX_m$ is an $A$-weighted Davenport Z-sequence  \textit{whp} (as $n\to\infty$).
   
 \end{enumerate}
 \end{thm} 
Note that in the last part of Theorem \ref{RandDav} (where the weight set is $A=\Z_n^*$) we do not have an analogous statement (as in the previous parts) with a lower bound for $m$ for $\mfX_m$ to not be a weighted Davenport $Z$-sequence {\it whp}; indeed, as we shall see later, if $p_1$ is sufficiently large, a random sequence of length $3$ is a $\Z_n^*$-weighted Davenport Z-sequence {\it whp}. This will follow as a consequence of our proof of Theorem \ref{RandDav}. 

An interesting aspect of our results is in the contrast between what we may dub the deterministic case versus the random case. For instance while $D(\Z_n)=n$,  a random sequence of size approximately $\log_2 n$ is already a Davenport Z-sequence with high probability. One may suspect that with a random $\Z_n$-sequence, the length of the sequence in order that it is    a weighted Davenport Z-sequence with high probability is significantly smaller than the corresponding length in the deterministic case. This is however not always true as evidenced by the second part of Theorem \ref{RandDav} corresponding to  the case $A=\{-1,1\}$ where the requisite sequence size drops down {\it merely by a constant factor}. But such instances may be more the exception than the rule. At the moment, we do not quite have a definitive answer to this question, but the overall motif seems unmistakeable.

The rest of the paper is organized as follows. In the following three sections we prove Theorems \ref{X_k(m)}, \ref{RandErd} and \ref{RandDav} respectively. The proof of the last part of Theorem \ref{X_k(m)} invokes a result of Szemer\'edi \cite{Sze} that settled a problem posed by Erd\H{o}s and Eggleston. While the paper \cite{Sze} settles a different conjecture of Erd\H{o}s and Heilbronn, the same proof can be suitably modified to also settle the Erd\H{o}s-Eggleston conjecture, and Szemer\'edi notes this as such in \cite{Sze}. We include a proof of this result in the appendix, for the sake of completeness since we were unable to find any published proof of the same. We make no claim to any originality towards {\it this} proof; only the presentation and perspective (which is implicit in Szemer\'edi's paper) is ours. We conclude the paper with some general remarks and a few open questions.
  
 \section{Proof of Theorem \ref{X_k(m)}}
We start with the following simple observation. For any finite abelian group $G$ and $A,B\subset G$ satisfying $|A|+|B|>|G|$ we necessarily have $A+B=G$, where $A+B:=\{a+b:a\in A,b\in B\}$. This follows easily since for any $x\in G$
 we have $|A|+|x-B|>|G|$, so $A\cap (x-B)\ne\emptyset$, and that implies $x\in A+B$. 
 \begin{proof} 
\begin{enumerate}
\item  In the proof of the first part, we shall deal with the case where $n$ is even, or $n$ is odd, separately.  We start with the odd case.
Let $\bfx = (a_1,a_1, \ldots, a_k,a_k,a_{2k+1}, \ldots, a_{2n+3})$ be a $\Z_{2n+1}$-sequence of length $2n+3$. In words, the elements $a_1,\ldots, a_k$ appear twice (for some $k$) while the elements $a_{2k+1},\ldots,a_{2n+3}$ appear once. In particular,  the $a_i$'s are pairwise distinct. 

Set $\ell= 2(n-k)+3$, and consider the sets $A = \{ a_1, a_2, \ldots, a_k, a_{2k+1},\ldots, a_{2k+(\ell-1)/2} \}$ and $B =\{a_1, a_2, \ldots, a_k, a_{2k+(\ell+1)/2},\ldots, a_{2k+\ell} \}$.
Clearly, $|A|= n+1$ and $|B| = n+2$.
Let $A'$ denote the set of all possible sums of $n$ different elements of $A$ and similarly, let $B'$ denote the set of all possible sums of  $n+1$ distinct elements of $B$.
Clearly, $|A'|= n+1 $ and $|B'| = n+2 $, so $|A'|+ |B'| = 2n+3 > 2n+1$, therefore by the observation at the beginning of this section, it follows that $A'+B'=\Z_{2n+1}$. Hence $0\in A'+B'$, or equivalently, $\bfx$ admits a non-trivial zero-sum subsequence of length $2n+1$. This establishes $\mathsf s^{(2)}(\Z_{2n+1}) \leq 2n+3$.

To show that $\mathsf s^{(2)}(\Z_{2n+1}) \geq 2n+3$, consider the sequence $$\bfx = (1,2,\ldots, n-1,n,n,n+1,n+1, n+2, n+3,\ldots, 2n).$$ It is easy to see that $\bfx$ is a sequence of length $2n+2$ satisfying $\sum_i x_i=0$ but
no element is equal to $0$. Moreover, no element appears in $\bfx$ more than twice, so this establishes that $\mathsf s^{(2)}(\Z_{2n+1}) \geq 2n+3$.
 Consequently, $\mathsf s^{(2)}(\Z_{2n+1}) = 2n+3$. 

For the even case, again as before, let  $\bfx = (a_1,a_1,\ldots,a_k,a_k,a_{2k+1},a_{2k+2},\ldots,a_{2n+2})$ be a $\Z_{2n}$-sequence of length $2n+2$ where each $a_i$ appears twice for $1\le i\le k$ and the elements $a_{2k+1},\ldots,a_{2n+2}$ appear exactly once, and so again, the $a_i$'s are pairwise distinct.  Let $\ell= 2(n-k)+2$, and  define subsets $$A = \{ a_1, a_2, \ldots, a_k, a_{2k+1},\ldots, a_{2k+\ell/2} \}, B =\{a_1, a_2, \ldots, a_k, a_{2k+ (\ell/2)+1}, \ldots, a_{2k+\ell} \} $$ 
of $\Z_{2n}$ and let $A'$ and $B'$ be the sets of sums of $n$ distinct elements of  $A$ and $B$ respectively.
Since $|A|= n+1$ and $|B| = n+1$, we have $|A'|= n+1$ and $|B'| = n+1 $, so that $|A'|+ |B'| = 2n+2 > 2n$. By the observation made earlier and the same argument as in the previous case, it follows that  $\mathsf s^{(2)}(\Z_{2n}) \leq 2n+2$ for $n\ge3$.

To complete the proof of the theorem, consider the sequence 
$\bfx = (1,2,3, 1,3, 5, 6,\ldots, n-1,n, n, n+1, n+2,\ldots, 2n-1)$ for $n\ge 5$ of length $2n+1$. It is easy to check that again the sum of all the elements of $\bfx$ equals $0$ but since $0$ itself is not in $\bfx$, it follows that $\mathsf s^{(2)}(\Z_{2n}) \geq 2n+2$ for $n\ge 5$. For $n=4$ and $n=3$ consider the sequence $\bfx=(1,2,3,1,3,4,5,6,7)$ and $\bfx= (1,2,1,2,3,4,5)$ respectively, and the same argument works with these sequences. This completes the proof.
\item We now turn to prove the second part of the theorem. We start with the proof of the upper bound which we shall prove by induction on $k.$

The case of $k=2$ is just a special case of the first part of Theorem \ref{X_k(m)} that was proved above. Suppose now that $k \geq 3,$ and suppose that the statement holds for values less than $k$.
 
Let $\bfa$ be a sequence of size $p+k,$ where each element
appears at most $k$ times. If no element in $\bfa$ has multiplicity $k$, then since $p+k >p+k-1$, we are through by induction, so we may assume that there is at least one element in $\bfa$ that appears $k$ times. For each $1\le i\le k$, let $\ell_i$ denote the number of distinct elements of $\bfa$ which appear precisely $k-i+1$ times, and set $L_i:=\sum_{j=1}^i \ell_j$. Write 
$$\bfa= \left(\underbrace{a_1,\cdots,a_1}_{k \text{ times}},\cdots, \underbrace{a_{L_1},\cdots,a_{L_1}}_{k \text{ times}}, \underbrace{a_{L_1+1},\cdots,a_{L_1+1}}_{k-1 \text{ times}},\cdots, \underbrace{a_{L_2},\cdots,a_{L_2}}_{k-1 \text{ times}},\cdots, \underbrace{a_{L_{k-1}+1},\cdots a_{L_{k}}}_{1\text{ time}}\right) $$ where the $a_i$ are distinct for  $1\leq i\leq L_k$ and 
$k\ell_1+(k-1)\ell_2+\cdots + \ell_k= p+k.$

Consider the sets $A_i$ consisting of those $a_j$ which appear at least $i$ times in $\bfa$. More precisely, let $A_i=\{a_1,a_2,\ldots,a_{\ell_1+\cdots+\ell_{k-i+1}}\},$ for $1\leq i\leq k$. Note that $|A_i|= \ell_1+\cdots+\ell_{k-i+1}$. As in the proof of the previous part,  $A_i'$ be the sum of all $(\ell_1+\ell_2+\cdots+\ell_{k-i+1}-1)$-sum of elements of $A_i$ for each $i$, so that $|A_i'|=|A_i|.$  By the
Cauchy-Davenport Theorem (see \cite{Nath}, theorem 2.3, page 44, for instance) we have

$$|\sum_{i=1}^{k}A_i'|\geq \text{min} (p, \sum_{i=1}^{k}|A_i'|-k+1 )
= \text{min } (p, p+k-k+1) =p.$$
In particular we have  $\sum_{i=1}^{k}A_i'=\Z_p$ and hence
$0 \in \sum_{i=1}^{k}A_i',$ where $A_i'$ is the sum of all $(\ell_1+\ell_2+\cdots+\ell_{k-i+1}-1)$ elements of the sequence. In particular, there is a subsequence of $\bfa$ of length $\sum_{i=1}^{k}(\ell_1+\ell_2+\cdots+\ell_{k-i+1}-1)= p+k -k = p$ whose sum is zero, and that completes the induction. 
 
To get the lower bound for $\mathsf s^{(k)}(\Z_{p})$,  write $p= (k-1)\ell +r$ for some $\ell$, and $0<r<k-1$, so that $\ell=\frac{p-r}{k-1}.$ Denote $p-1$ by $-1$. We shall construct a sequence $\bfx$ of length $p+k-1$ of the form 
$$\bfx=\left(\underbrace{-1,\ldots,-1}_{k \text{ times}}, \underbrace{0,\ldots,0}_{k-1 \text{ times}},\ldots,\underbrace{\ell-2,\ldots,\ell-2}_{k-1 \text{ times}},x_1,\ldots,x_{r+k-2}\right)$$ where $x_1,\ldots,x_{r+k-2}$ are pairwise distinct such that $\bfx$ has no zero-sum subsequence of length $p$. 

Towards that end, we shall show that if $p$ is sufficiently large (we will make this more precise soon) there exists a subset $\{x_1,\ldots,x_{r+k-2}\}\subset\{0,1\ldots,\ell-2\}$ in $\Z_p$ such that
 \begin{eqnarray}\label{sumtill l-2}\sum_{i=1}^{r+k-2}x_i= -(k-1)(1+2+\cdots+(\ell-2)) &=& -\frac{(p-r-(k-1))(p-r-2(k-1))}{2(k-1)}\end{eqnarray}
 in $\Z_p$. 
 
 First, to see why this suffices for our needs, observe that the sequence $\bfx$ has no element appearing more than $k$ times. If (\ref{sumtill l-2}) holds then the sum of all the elements of $\bfx$ equals $-k$. Denote the subsequence of $\bfx$ omitting the $-1$'s by $\bfa$, so that $\bfa$ is a $\Z_p$-sequence of length $p-1$.  Then note that the sum of any $s\le k-1$ elements of $\bfa$ (when viewed as integers) is at most $(k-1)(\ell-2)<p-k$. Consequently, for any $J\subset [1,p-1]$
with $|J|=s \le k-1$ we have $\bfa_J\ne -(s+1)$.  The upshot of this observation is that if $\bfx$ admits a zero-sum subsequence of length $p$, then its complementary subsequence, i.e., the subsequence consisting of the remaining elements in $\bfx$, is a subsequence of length $k-1$ which has sum $-k$. This complementary sequence consists of some subsequence $\bfb$ of $\bfa$ of length $s$ (for some $1\le s\le k-1$) along with some $k-1-s$ elements that are equal to $-1$.  But this implies that the sum of the elements of $\bfb$ must equal $-(s+1)$ which is not possible by the observation made above. So, it suffices to find $x_1,\ldots,x_{r+k-2}\subset\{0,\ldots,\ell-2\}$ such that (\ref{sumtill l-2}) holds. 

Suppose $\ell$ is even, say $\ell = 2t$ so that $t=\frac{p-r}{2(k-1)}$. A straightforward calculation shows that  (\ref{sumtill l-2}) gives us
 \begin{eqnarray*}\sum_{i=1}^{r+k-2} x_i &=& -\frac{(p-r-(k-1))(p-r-2(k-1))}{2(k-1)} \\&=&(r+k-1)(t-1)
 \end{eqnarray*} 
  in $\Z_p$. Set $x_{r+k-2} = 2t-(k-1), x_{r+k-3}=2t-(k-2), x_{r+k-4}=k-2-r$. If $r+k$ is odd then for $1\le i\le \frac{r+k-5}{2}$ set
 \begin{eqnarray*} x_{2i-1}&=&t - i,\\ x_{2i}&=&t+i\end{eqnarray*} 
 If $p\ge (k-1)(2k-1)$, then $t+\frac{r+k-5}{2}\le\ell-2$ and $0\le k-2-r<t-\frac{r+k-5}{2}$, so all these choices for $x_i$ are pairwise distinct and lie in the set $\{0,\ldots,\ell-2\}$. It is a straightforward check to see that $\displaystyle\sum_{i=1}^{r+k-2} x_i = (r+k-1)(t-1)$.  
 
 If $r+k$ is even, then set $x_{r+k-5}=t$ for  $1\le i\le \frac{r+k-6}{2}$ set
 \begin{eqnarray*}
 x_{2i-1}&=&t - i,\\ x_{2i}&=&t+i.\end{eqnarray*}
Again, it is straightforward to check that the $x_i$ satisfy the requirements, and this settles this case.
 
 If $\ell$ is odd, write $\ell=2t+1$ so that $t=\frac{p-(r+k-1)}{2(k-1)}$. In this case, (\ref{sumtill l-2}) simplifies as  
 \begin{eqnarray*}\sum_{i=1}^{r+k-2} x_i&=& -\frac{(p-r-(k-1))(p-r-2(k-1))}{2(k-1)}\\ &=&t(r+2(k-1))\\ &=&tr + p-(r+k-1)\\ &=& (t-1)r -k+1. \end{eqnarray*} in $\Z_p$.  Again, to satisfy these requirements, set $a=\left\lfloor\frac{(t-2)r}{r+k-3}\right\rfloor$, and set $x_i=a-1-i$ for $1\le i\le r+k-3$. Observe that  if $p=\Omega(k^4)$ then $r+k\le a\le t-2<\ell$, and  
 $$\sum_{i=1}^{r+k-3}x_i=(t-2)a - \frac{(r+k-3)(r+k-2)}{2}-s(r,k)$$ for some $0\le s(r,k)<r+k-3\le 2k-5$.  Now set 
 $$x_{r+k-2}= \frac{(r+k-3)(r+k-2)}{2}+2(r-1)-s(r,k)$$
 so that $\sum_i x_i = (t-1)r -k+1$. It is easy to see that for all integers $k$, $\frac{(k-2)(k-1)}{2}\ge 2k-5$, so $$\frac{(r+k-3)(r+k-2)}{2}+2(r-1)-s(r,k)>\frac{(k-2)(k-1)}{2}-(2k-5)\ge 0$$ which implies that $x_{r+k-2}> 0$. Also,
 $$x_{r+k-2}+(r+k-2)\le 2(2k-5)+(2k-5)(k-2)\le\frac{t-2}{k-3}\le a$$
 if $p\ge \Omega(k^4)$. These inequalities establish that  $x_i\in\{0,\ldots,\ell-2\}$ for all $1\le r+k-2$, are pairwise distinct, and  (\ref{sumtill l-2}) holds; this completes the proof.
  
 \textbf{Remark:} Our construction of the sequence $\bfx$ seems somewhat ad hoc,  and unlike some of the results to the Inverse Sum Problems (see \cite{GGS}), these extremal sequences do not appear to be unique in any sense. Also, 
as we make no attempt to optimize for $p_0(k)$ (in the statement of part 2 of Theorem \ref{X_k(m)}), it should be possible to find better examples than ours;  as noted earlier, we may take $p_0(k)=\Omega(k^4)$ but there may be much sharper bounds for $p_0(k)$. Another side to this argument is that for any $k\le O(p^{1/4})$ we have $\mathsf s^{(k)}(\Z_{p})=p+k$.

\item Before we prove the third part of Theorem \ref{X_k(m)} we state a theorem due to Szemer\'edi, which settled a conjecture of Erd\H{o}s and Eggleston. In order to state this, we need a definition. For a finite abelian group $G$ and $A\subset G$, by $\Sym(A)$ we mean the set of all those elements that occur as a sum of elements of some non-trivial subset of $A$. 
\begin{thm}\label{ErdEgg} There exists an absolute constant $1/2>\ep_0>0$ such that\footnote{The statement of Szemer\'edi's theorem in fact only states the existence of a constant $0<\ep_0<1$ but putting a smaller constant does not change the statement in any way, so we shall put a bound $\ep_0\le 1/2$ here for convenience.}
 the following holds. If $G$ is a finite abelian group and  $A\subset G$ then either $0\in \Sym(A)$ or $|\Sym(A)|\ge\ep_0|A|^2$.\end{thm}
A conjecture of Erd\H{o}s-Heilbronn, which was settled by Szemer\'edi \cite{Sze} states that there exists an absolute constant $C\ge 1$ such that for any abelian group $G$ of order $n$, and any subset $A\subset G$ with $|A|\ge C\sqrt{n}$, there exists some non-trivial subset of $A$ the sum of whose elements equals zero. As we mentioned in the introduction, Szemer\'edi remarks \cite{Sze} in his paper that the same methods actually can be extended to prove the aforementioned result as well. As mentioned in the Introduction, the proof of Theorem \ref{ErdEgg} appears in Appendix I.

We will also need another result due to Scherk, settling a problem proposed by L. Moser. The original version was for the case $k=2$ but a simple induction (on $k$) establishes the following version as well.
\begin{thm}\label{Shrk} (\cite{Scherk}, \cite{Yu})  Suppose $B_1,\ldots,B_k\subset\Z_n$  such that $0\in \cap_{i=1}^k B_i$, and suppose the equation $0=b_1+\cdots+b_k$ with 
$b_i\in B_i$ has the unique solution $b_i=0$ for all $i$, then $$|\sum_{i=1}^k B_i|\ge \min\left\{ \sum_{i=1}^k |B_i|-k+1, n\right\}.$$\end{thm}
We are now in a position to prove the third part of Theorem \ref{X_k(m)}.  We start with the lower bound.  

Consider the sequence $\bfx=(\underbrace{1,\cdots,1}_{k \text{ times}},\underbrace{2,\cdots,2}_{k \text{ times}},\cdots, \underbrace{t,\cdots,t}_{k \text { times}}),$ where $\frac{kt(t+1)}{2} < \frac{n}{2}.$ Set $N=kt$ so that $N=\Omega(\sqrt{nk})$.
 Then note that for any non-trivial $I\subset [N]$,  $0<\bfx_I<n/2$ by choice, so $\bfx$ does not admit a zero-sum subsequence. 
  
  Let $C\ge 1$ be the constant from Szemeredi's theorem settling the Erd\H{o}s-Heilbronn conjecture, and let $\ep_0$ be the constant from Theorem \ref{ErdEgg}. Let $C^*=2C/\ep_0$. We claim that if $\bfx$ is a $k$-restricted $\Z_n$-sequence of length $m \ge \lceil C^*\sqrt{nk}\rceil +1$, then $\bfx$ is a Davenport Z-sequence, and we shall prove this by induction on $k$.  The case $k=1$ simply follows from Szemer\'edi's theorem, so suppose $k>1$ and suppose the result holds for smaller values than $k$. 
  
Let $x_0$ be a non-zero element of the sequence $\bfx$ and let  $\bfy$ be the sequence obtained from $\bfx$ by removing $x_0$. Then $\bfy$ is a sequence of length at least $\lceil C^*\sqrt{nk}\rceil\ge C^*\sqrt{nk}$.    Write  
$$\bfy=(\underbrace{x_1,\ldots,x_1}_{\ell_1 \text{ times}},\ldots,\underbrace{ x_{r},\ldots,x_{r}}_{\ell_r \text{ times}})$$ with $k\ge\ell_1\ge\cdots\ge\ell_r$ and where $x_1,\ldots,x_r$ are pairwise distinct non-zero elements of $\Z_n$.
 If $k>\ell_1$, then we are through by induction since $\lceil C^*\sqrt{nk}\rceil \ge \lceil C^*\sqrt{n(k-1)}\rceil +1$ and no element of  $\bfy$ appears more than $k-1$ times. So, we may assume that $\ell_1=k$.
 
For  $1\le i \le k$ let $R_i$ denote the set of those $x_j$ that appear at least $i$ times. Clearly $|R_1|\ge\cdots\ge|R_k|\ge 1$. Let $r_i=|R_i|$. If $r_1\ge C\sqrt{n}$ then again, by Szemer\'edi's Theorem $0\in\Sym(R_1)$, and we are through,  so we may assume that $r_1<C\sqrt{n}$. 

 If $r_k \leq \frac{C^*\sqrt{n}}{\sqrt{k}+\sqrt{k-1}}$ then the subsequence of $\bfy$ obtained by removing the elements of $R_k$ is a $(k-1)$-restricted subsequence of length at least $C^*\sqrt{nk}-r_k\ge C^*\sqrt{n(k-1)}$, so again, by induction, $\bfy$ is a Davenport Z-sequence. So, again, we may assume that $$r_k> \frac{C^*\sqrt{n}}{\sqrt{k}+\sqrt{k-1}}\ge (C/\ep_0) \sqrt{\frac{n}{k}}.$$  
 Let $ B_i = \Sym(R_i)\cup\{0\}$ for $i=1,\ldots,k.$  By Theorem \ref{ErdEgg} either $0\in \Sym(R_i)$ in which case we are through, or $|B_i| \geq \ep_0|R_i|^2$ for each $i$.
 If $0$ can be written as $b_1+\cdots+ b_k$ with $b_i\in B_i$ with at least one of the $b_i\ne 0$,  then again, we are through, so we may assume that $0=b_1+\cdots+b_k$ with $b_i\in B_i$ implies that $b_i=0$ for each $i$. Then by the result of Scherk (Theorem \ref{Shrk}), 
 $$|\sum_{i=1}^kB_i| \geq \min\left(n, \sum_{ i=1}^k\ep_0|R_i|^2 -k+1\right) = n $$ as $|R_i|\ge (C/\ep_0) \sqrt{n/k}$ for all $i$ and $\frac{C^2}{\ep_0}>2$.  Therefore, $\sum_{i=1}^kB_i = \Z_n,$ which implies that $-x_0 = \sum_{i=1}^kb_i$ for some choices of  non-zero $b_i \in B_i.$ In particular, $\bfx$ is a Davenport Z-sequence as required. This completes the induction and the proof.
  \end{enumerate}
 \end{proof}
 
\section{Proof of Theorem \ref{RandErd}}
Before we get into the proofs, we set up a bit of notation here which will be used repeatedly. By $\bI(x)$ we mean the indicator function which equals $1$ if $x=0$ and is zero otherwise. For a subset $I\subset [m]$, we shall denote by $\mfX_I$, the sum $\mfX_I:=\sum_{i\in I} X_i$.  We shall sometimes write $\mfX(I)$ instead of $\mfX_I$ for notational convenience.

Before we launch into the proofs, we state a simple lemma.
\begin{lem}\label{pwi}
\begin{enumerate}
\item[(a)] Let $X_1,\ldots,X_m$ be uniformly distributed on $\Z_n$ and suppose that for any $i\ne j$, $X_i, X_j$ are independent. Then, for any $a_1,\ldots, a_m \in \Z_n^* $ 
the random variable $\sum_{i=1}^m a_iX_i$ is uniformly distributed on $\Z_n$. 
\item[(b)] Suppose $\mfX_m=(X_1,\ldots,X_m)$ where the $X_i$ are independent and uniformly distributed over $\Z_n$.  Let $I,J\subset [m]$ be non-empty subsets of indices, and let $Y_1:=\sum_{i\in I} a_iX_i$ and $Y_2:=\sum_{j\in J} b_j X_j$ for some $a_i,b_j\in\Z_n^*$. Let $i_0\ne j_0$ be distinct indices in $I\cup J$, and let $x_i\in \Z_n$ for all $i\in I\cup J\setminus\{i_0,j_0\}$ be arbitrary. If the system of equations 
\begin{eqnarray}\label{uniquesoln} Y_1=\alpha, Y_2=\beta, X_i=x_i\textrm{\ for\ all\ } i\neq i_0,j_0\end{eqnarray} admits at most $r$ solutions for the pair $(X_{i_0}, X_{j_0})$ for any $\alpha,\beta, x_i$, then $\bP(Y_1=\alpha, Y_2=\beta)\le\frac{r}{n^2}$. In particular, if the system has a unique solution for the pair $(X_{i_0}, X_{j_0})$, then $Y_1$ and $Y_2$ are independent. 
\end{enumerate}
\end{lem}
\begin{proof} (Of Lemma \ref{pwi}) Both the parts work by a conditioning argument. For the first part, pick an arbitrary $\Z_n$-sequence $\bfx=(x_2,\ldots,x_m)$ and condition on $X_i=x_i$ for $i\ge 2$. Then $$\bP\left(\sum_i a_iX_i=\alpha\mid X_i=x_i\textrm{ for } i\ge 2\right) = \bP\left(X_1=\frac{1}{a_1}\left(\alpha-\sum_{i>1} a_ix_i\right)\right)=\frac{1}{n}$$ and this holds for all choices of the sequence $\bfx$. Hence the same holds for the unconditional probability as well.

For the second part,  let $\bfx$ be an arbitrary $\Z_n$-sequence indexed by the elements of $T:=I\cup J\setminus\{i_0,j_0\}$ and condition on $X_{\ell}=x_{\ell}$ for all $\ell\in T$. Let the number of solutions for $(X_{i_0}, X_{j_0})$ to the system of equations (\ref{uniquesoln}) be $R(\alpha,\beta,\bfx)$.  
Then
$$\bP\left(Y_1=\alpha, Y_2=\beta\mid X_{\ell}=x_{\ell}\textrm{ for all }\ell\in T\right)=\frac{R(\alpha,\beta,\bfx)}{n^2}\le \frac{r}{n^2}$$
by the hypothesis.  Since this holds for all choices of $\alpha, \beta, \bfx$, the same inequality holds unconditionally as well, and that completes the proof.\end{proof}
\begin{proof}  \begin{enumerate}
 \item Let $\mfX_{n+2}=(X_1,\ldots,X_{n+2})$ be a random $\Z_n$-sequence.   Let 
 $\Hy:=\{ I\subset[n+2]: |I|= n \}$ and $N:=\sum_{I\in \Hy} \mathbb{I}(\mfX_I)$.
Then	
 $$\mathbb{E}(N) = \sum_{I\in \Hy}\mathbb{P}(\mfX_I=0) =\frac{1}{n} \binom{n+2}{n} = \frac{(n+2)(n+1)}{2n}=\Omega(n),$$
	and, $$\Var (N) = \sum_{I\in \Hy}\Var (\mathbb{I}(\mfX_I)) +  \sum_{\substack{I\ne J\\ I,J\in \Hy}} \Cov(\mathbb{I}(\mfX_I),
	\mathbb{I}(\mfX_J)).$$
	 Pick $i\in I\setminus J$ and $j\in J\setminus I$, and consider the equations $\mfX_I=\alpha, \mfX_J=\beta$, and $X_{\ell}=x_{\ell}$ for all $\ell\neq i,j$. It is straightforward to check that 
	this system of equations admits a unique solution for $(X_i,X_j)$, so by Lemma \ref{pwi}(b) it follows that $\{\mfX_I\}_{I\in \Hy}$ are pairwise independent. Consequently,  
	$\Cov (\mathbb{I}(\mfX_I),\mathbb{I}(X_J))=0$ for $I\ne J\in \Hy$. Also, $\Var (\mathbb{I}(\mfX_I))=\frac{1}{n}(1- \frac{1}{n}),$ so 
	$$\Var (N) = \sum_{I\in \Hy} \Var (\mathbb{I}(\mfX_I))=\frac{1}{n}\left(1- \frac{1}{n}\right)\frac{(n+2)(n+1)}{2}=O(n).$$ 
	Therefore, by Chebyshev's inequality we have,
	$$\bP(N=0)\le \mathbb{P}(|N-\mathbb{E}(N)| \ge \mathbb{E}(N)) \le \frac{\Var(N)}{
	(\mathbb{E}(N))^2} = \frac{\frac{1}{2}(1-\frac{1}{n})}{\frac{1}{4}
	(1+\frac{2}{n})(n+1)}=O\left(\frac{1}{n}\right).$$
	Hence $\mathbb{P}(N > 0)\rightarrow 1$ as $n\to\infty$.
\item Let $\mfX_n=(X_1,\ldots,X_n)$ be a random $\Z_n$-sequence and $A = \{ a,b \}$. Without loss of generality, assume $\{a,a+b,a-b\} \subset \Z_n^*$. 
	Let $\Hy:=\{I \subset [n]:\emptyset\subsetneq I\subsetneq [n]\}$ and let $$N:=\sum_{I\in \Hy} \mathbb{I}_{a\mfX(I)+b\mfX(\overline{I})}$$ where $\overline{I}:=[n]\setminus I$.
	
It is easy to see that  $$\mathbb{E}(N) = \sum_{I\in\Hy}\mathbb{P}\Bigg(a\mfX_I+b\mfX_{\overline{I}}=0\Bigg) =\frac{2^n-2}{n}$$ as each of the probabilities in the summation above equals $1/n$ by Lemma \ref{pwi}(a).  Again,  $$\Var (N) = \sum_{I\in \Hy}\Var(\mathbb{I}_{a\mfX(I)+b\mfX(\overline{I})}) + 
	 \sum_{\substack{I\ne J\\ I,J\in\Hy}} \Cov(\mathbb{I}_{a\mfX(I)+b\mfX(\overline{I})},\mathbb{I}_{a\mfX(J)+b\mfX(\overline{J})})$$
	and it follows that $$\Var\left(\mathbb{I}_{a\mfX(I)+b\mfX(\overline{I})}\right) = \frac{1}{n}\left(1- \frac{1}{n}\right).$$  
	We claim that 
	\begin{eqnarray}\label{abcase} \Cov(\mathbb{I}_{a\mfX(I)+b\mfX(\overline{I})},\mathbb{I}_{a\mfX(J)+b\mfX(\overline{J})})=0\textrm{\ for }I\ne J.\end{eqnarray}
	 Hence
	$$\Var (N) = \sum_{I\in \Hy} \Var(\mathbb{I}_{a\mfX(I)+b\mfX(\overline{I})}) < \frac{2^n}{n}$$  so again by
	Chebyshev's inequality, 
	$$\bP(N=0)\le \mathbb{P}(|N-\mathbb{E}(N)| \ge \mathbb{E}(N)) \le O\left(n2^{-n}\right)$$ 
	which gives us what we seek.
	
 So to complete the proof, it remains to establish (\ref{abcase}) which amounts to showing that    	
	$$\mathbb{P}\Bigg[\big(\mfX(I)=-(b/a)\mfX(\overline{I})\big)\textrm{ and } \left(\mfX(J)=-(b/a)\mfX(\overline{J})\right)\Bigg]=\frac{1}{n^2}$$ whenever $I\ne J$.
	
	Towards this end, suppose $I\ne J$ be subsets of $[m]$. We shall appropriately choose $i\ne j$ in $I\cup J$ and consider the equations  $\mfX_I=\alpha, \mfX_J=\beta$, and $X_{\ell}=x_{\ell}$ for all $\ell\neq i,j$. We shall use Lemma \ref{pwi}(b) and show that this system of equations has a unique solution for $(X_i,X_j)$. We have the following cases:
\begin{enumerate}
\item[i.] $I\setminus J\ne\emptyset$ and $J\setminus I\ne\emptyset$: Pick $i\in I\setminus J$, $j\in J\setminus I$. The system of equations  reduces to two equations of the form $aX_i+bX_j=\xi_1,bX_i+aX_j=\xi_2$ for some $\xi_1,\xi_2\in\Z_n$ and since by assumption $a^2-b^2\in\Z_n^*$, these equations admit a unique solution for $(X_i,X_j)$. Hence we are through by Lemma \ref{pwi}(b).  
\item[ii.] If the previous case does not arise, then either $I\subset J$ or $J\subset I$. Suppose without loss of generality that $I\subset J$.  Pick $i\in I, j\in J\setminus I$. Again, the system of equations reduces to two linear equations of the form $aX_i+bX_j=\xi_1,aX_i+aX_j=\xi_2$ for some $\xi_1,\xi_2\in\Z_n$. Again, since $a,a-b\in\Z_n^*$, it is easy to see that the system of equations admits a unique solution for $(X_i,X_j)$ and so by Lemma \ref{pwi}(b) we are through. \end{enumerate}
\item Let $\mfX_{n+1}=(X_1,\ldots,X_{n+1})$ be a random $\Z_n$-sequence, and define 
\begin{eqnarray*}\Hy&:=&\{(\{I,J\}: I,J\subset[n+1],  I,J\ne\emptyset,I\cap J =\emptyset, |I\cup J| = n \},\\
	                                   N &:=& \sum_{\{I, J\}\in \Hy} \mathbb{I}(\mfX_{I,J}),\end{eqnarray*} where
 	$\mathbb{I}(\mfX_{I,J}):=\bI(\mfX_I-\mfX_J)$.  Note that $|\Hy|=(n+1)(2^{n-1}-1)$.
	
	Again, \begin{eqnarray*} \mathbb{E}(N) = \sum_{\{I,J\}\in \Hy}\mathbb{P}(\mfX_I=\mfX_J) = \frac{(n+1)(2^{n-1}-1)}{n}\end{eqnarray*} for the same reason as in the preceding discussions.
Again we shall bound $\Var(N)$ but unlike the proof of parts 1 and 2 of Theorem \ref{RandErd}, we do not always have pairwise independence of the random variables $\mfX_{I,J}$ in this case. But it turns out that most pairs $\{I,J\}\ne\{I',J'\}$ are pairwise independent and that is sufficient for our purpose here.  
	
Suppose $\{I,J\}\ne \{I',J'\}$ are pairs in $\Hy$ such that $I\cup J=[n+1]\setminus\{a\}$ and $I'\cup J'=[n+1]\setminus\{b\}$ with $a\ne b$.   We claim that the random variables $\mfX_{I,J}$ and $\mfX_{I',J'}$ are independent. Without loss of generality, suppose $a\in I'$ and $b\in I$. Let $T=[n+1]\setminus\{a,b\}$ and let $\bfx$ be an arbitrary $\Z_n$-sequence indexed by the elements of $T$.  The system of equations $\mfX_I-\mfX_J=\alpha, \mfX_{I'}-\mfX_{J'}=\beta, X_{\ell}=x_{\ell}$ for all $\ell\in T$ reduces to 
	$X_a=\bfx_{J'}-\bfx_{I'\setminus\{a\}}$ and $X_b=\bfx_J-\bfx_{I\setminus\{b\}}$, thereby giving a unique solution for $(X_a,X_b)$; consequently, by Lemma \ref{pwi}(b) $\mfX_{I,J},\mfX_{I',J'}$ are independent. 
	
The aforementioned argument in fact shows that if $n$ is odd, then $\mfX_{I,J}$ are {\it all} pairwise independent. Indeed, suppose $I\cup J= I'\cup J'=[n]$ (without loss of generality). Since at least one of $I,J$ meets both $I'$ and $J'$ nontrivially, suppose $I$ intersects both $I', J'$ and pick $i\in I\cap I'$ and $j\in I\cap J'$. Then for $T=[n+1]\setminus\{i,j\}$ and any $\Z_n$-sequence $\bfx$ indexed by the elements of $T$, the same system of equations as before reduces to $X_i+X_j=\xi_1,X_i-X_j=\xi_2$ for some $\xi_1,\xi_2\in\Z_n$. But observe that if $n$ is odd, this system of equations admits a unique solution for $(X_i,X_j)$ and so again by Lemma \ref{pwi}(b), it follows that $\mfX_{I,J},\mfX_{I',J'}$ are independent. 
	
	If $n$ is even, then this might not admit any solution at all, in which case $\mathbb{I}(\mfX_{I,J}),\mathbb{I}(\mfX_{I',J'})$ are negatively correlated. If $\xi_1+\xi_2$ is even, then the pair of linear equations above admit two possible solutions for $X_i$ (say), and for each of these, a unique value for $X_j$. Consequently, in these cases, the system of equations considered above admits at most $2$ solutions $(X_i,X_j)$ {\it for any choice of} $\alpha,\beta,\bfx)$. So again, by Lemma \ref{pwi}, it follows that 
	$$\Cov(\mathbb{I}(\mfX_{I,J}),\mathbb{I}(\mfX_{I',J'}))=\bP(\mfX_{I,J}=0\textrm{ and } \mfX_{I',J'}=0)-\bP(\mfX_{I,J}=0)\bP(\mfX_{I',J'}=0)\le \frac{2}{n^2}-\frac{1}{n^2} = \frac{1}{n^2}.$$
	
	To complete the proof, we need a bound on the number of distinct pairs $\{I,J\},\{I',J'\}$ of members of $\Hy$ such that $I\cup J=I'\cup J'$. Since there are $n+1$ choices for the element not in $I\cup J$, $2^n-2$ choices for $I$ and  $2^n-3$ further choices for $I'$, we have 
	\begin{eqnarray*}\Var (N) &=& \sum_{\{I,J\}\in \Hy} \Var(\mathbb{I}(\mfX_{I,J})) +  2\sum_{\substack{\{I,J\}\neq\{I',J'\}\\\{I,J\},\{I',J'\}\in\Hy}} \Cov(\mathbb{I}(\mfX_{I,J}),
	\mathbb{I}(\mfX_{I',J'}))\\
	&=& \frac{1}{n}\left(1- \frac{1}{n}\right)(n+1)(2^{n-1}-1) +  2\frac{(n+1)(2^n-2)(2^n-3)}{n^2},\\
	&\le& \frac{3\cdot 2^{2n}}{n}\end{eqnarray*}
	so, by Chebyshev's inequality,
	$$\bP(N=0)\le\mathbb{P}(|N-\mathbb{E}(N)| \ge \mathbb{E}(N)) \le \frac{\Var(N)}{
	(\mathbb{E}(N))^2} = O\left(\frac{1}{n}\right),$$ so again, we have $\mathbb{P}(N > 0)\rightarrow 1.$
	This completes the proof of this part and that of the theorem as well.
\end{enumerate}\end{proof}
 
 \textbf{Remark:} Since one needs a sequence of size at least $n$ in order that it is an Erd\H{o}s Z-sequence, the results of the previous theorem assert that for random $\Z_n$-sequences one does not need much more than the {\it absolute} minimum required size for it to be an Erd\H{o}s Z-sequence \textit{whp}. In particular, if $n$ is prime, then for any $A$ of size at least $3$, a random $\Z_n$-sequence of length $n$ is \textit{whp} an $A$-weighted Erd\H{o}s Z-sequence. 
  
 \section{Proof of Theorem \ref{RandDav}}
 Let $\mfX_m=(X_1,\ldots,X_m)$ be a random $\Z_n$-sequence. The proof of the first and second parts of this theorem again involve the application of Chebyshev's inequality. The third part needs an additional lemma which we shall state at the appropriate juncture. 
  
 Let $\omega(n)$ be an arbitrary function satisfying $\omega(n)\to\infty$ as $n\to\infty$.
\begin{proof}
 \begin{enumerate}
 \item We start with the lower bound. Fix $m\le \log_2 n - \omega(n)$, and let $\Hy_m:=\{I: I\subset [m], I\ne\emptyset \}$ and $N_m:= \sum_{I\in \Hy_m} \mathbb{I}(\mfX_I)$.
	
	Then $$\mathbb{E}(N_m) = \sum_{I\in \Hy_m}\mathbb{P}(\mfX_I=0) = \frac{2^m-1}{n} \le \frac{1}{2^{\omega(n)}}$$ and consequently it follows - by Markov's inequality - that $N_m=0$ \textit{whp}.	
	
	For the upper bound, let $m \ge \log_2 n +\omega(n)$; then 
	$$\Var (N_m) = \sum_{I\in \Hy_m}
	\Var (\mathbb{I}(\mfX_I)) +  \sum_{\substack{I\ne J\\I,J\in\Hy_m}} \Cov(\mathbb{I}(\mfX_I),\mathbb{I}(\mfX_J)).$$
By a very similar argument to that in the proof of part 1 of Theorem \ref{RandErd}, it follows that  $\mfX_I$'s are pairwise independent, so that 
	$ \Cov(\mathbb{I}(\mfX_I),\mathbb{I}(\mfX_J))=0$ for $I\ne J\in \Hy_m$. Hence, 
	$$\Var(N_m) = \sum_{I\in \Hy_m} \Var(\mathbb{I}(\mfX_I))=\frac{1}{n}\left(1- \frac{1}{n}\right)(2^m-1)\le \frac{2^{(\log_2n+\omega(n))}}{n},$$
 so by Chebyshev's inequality
	$$\bP(N_m=0)\le\frac{\Var(N_m)}{\bE^2(N_m)} \le \frac{1}{2^{\omega(n)}}.$$
	This completes the proof of the first part.
	
 \item   Note that  if $\bfx=(x_1,\ldots,x_m)$ is a $\{1,-1\}$-weighted Davenport Z-sequence, then there are disjoint subsets $I,J$ with $I\cup J\ne\emptyset$ (with one of them possibly empty) such that $\bfx_I=\bfx_J$. In view of this, define 
 \begin{eqnarray*} \Hy_m&:=&\{\{I,J\}: I,J\subset [m], I\cup J\ne\emptyset, I\cap J=\emptyset\},\\
                              N_m&:=& \sum_{\{I,J\}\in \Hy} \mathbb{I}(\mfX_{I,J}),\end{eqnarray*} 
     where again $\mathbb{I}(\mfX_{I,J}):=\bI(\mfX_I-\mfX_J)$.  Note that we allow for $I$ (resp. $J$) to be the empty set in which case $\mfX_I$ (resp $\mfX_J$) is zero.
     
First, observe that $|\Hy_m|=\frac{3^m-1}{2}$. Indeed, the number of ordered pairs of sets $(I, I\cup J)$ where both are not empty is $3^n-1$ and these ordered pairs count each member of $\Hy_m$ twice. Hence if $m\le\log_3 n-\omega(n)$,                 
  $$\mathbb{E}(N_m) = \sum_{(I,J)\in \Hy_m}\mathbb{P}(\mfX_I=\mfX_J) = \frac{3^m-1}{2n}<\frac{1}{3^{\omega(n)}}$$
  which implies $\mathbb{P}(N_m >0) \rightarrow 0$, and this establishes that  {\it whp} $\mfX_m$ is not a $\{1,-1\}$-weighted Davenport Z-sequence.
 
 For the other part, suppose $m\ge\log_3 n+\omega(n)$. As before we bound
   $$\Var (N_m) = \sum_{\{I,J\}\in \Hy_m}	\Var(\mathbb{I}(\mfX_{I,J}))+ \sum_{\substack{\{I,J\},\{I',J'\}\in\Hy_m\\\{I,J\}\ne\{I',J'\}}} \Cov(\mathbb{I}(\mfX_{I,J}),\mathbb{I}(\mfX_{I',J'}))$$                         
 by using Lemma \ref{pwi}. The following claim is crucial: 
 \begin{claim}\label{Uncorr}
 $\Cov(\mathbb{I}(\mfX_{I,J}),\mathbb{I}(\mfX_{I',J'}))=0$ unless, one of the following two cases arise: 
 \begin{enumerate}
 \item $I=\emptyset, I',J'\ne\emptyset$ and $J=I'\cup J'$.
 \item  $I, J,I',J'\ne\emptyset$, $I\subset J'\subset I\cup J$ and $I'\subset J$.\end{enumerate} In these cases, $$\Cov(\mathbb{I}(\mfX_{I,J}),\mathbb{I}(\mfX_{I',J'}))\le\frac{1}{n^2}.$$\end{claim}
 The proof of the claim while somewhat straightforward, is a study of several cases that arise, so we relegate it to  Appendix II.
 
Call a pair $\{I,J\},\{I',J'\}\in \Hy_m$ bad if one of the two cases above holds.  Note that the number of bad pairs $\{I,J\},\{I',J'\}\in\Hy_m$ of the first type is at most $4\cdot 3^m$ and the number of bad pairs of the second type is at most $4\cdot 5^m$. Consequently,  
 \begin{eqnarray*} \bP(N_m=0)\le\frac{\Var(N_m)}{\bE^2(N_m)}\le \frac{1}{\bE(N_m)} + O\left(\frac{5}{9}\right)^m=o(1)\end{eqnarray*}
  and as before, that completes the proof.

 \item We start with a lemma. For integers $x_1,\ldots, x_m$ we denote their gcd by $\textrm{gcd}(x_1,\ldots,x_m)$.
 \begin{lem}\label{relprime} Suppose $n=\displaystyle\prod_{i=1}^{r}p_i,$ where  $p_1<\cdots <p_r$ are odd primes and let $m\ge 2$ be a positive integer.  For any $x_1,\ldots,x_m$ there exist non-zero integers $a_1,\ldots,a_m\in\Z$ such that $\textrm{gcd}(a_i,n)=1$ for all $1\le i\le m$, and satisfying $\sum_{i=1}^m a_ix_i=\textrm{gcd}(x_1,\ldots,x_m)$.\end{lem}
 \begin{proof} (Proof of Lemma \ref{relprime}) We proceed by induction on $m$.  We first start with the case that $\textrm{gcd}(x_1\ldots,x_m)=1$ and then move to the general case later. Suppose $m=2$, and we have relatively prime integers $x_1,x_2$. We shall establish the lemma in this case by induction on $r$. Suppose $r=1$. In particular, there are integers $a_1,a_2\in\Z$ such that $a_1x_1+a_2x_2=1$. We claim that there exists $\lambda\in\Z$ such that both $a_1^*=a_1+\lambda x_2$ and $a_2^*=a_2-\lambda x_1$ are relatively prime to $n=p_1$. This would complete the proof in this case since $a_1^*x_1+a_2^*x_2=1$.  
 
Suppose  that $p_1$ divides $x_1$.  Then $\textrm{gcd}(p_1, a_2)=\textrm{gcd}(p_1,x_2)=1$. In particular, $a_2, x_2\ne 0$ in $\Z_{p_1}$, so $a_1+\lambda x_2\ne 0$ if and only if $\lambda\neq -(a_1/x_2)$ in $\Z_{p_1}$. Hence there is an available choice for $\lambda$ that satisfies the claim. The same argument works if $p_1$ divides $x_2$. So now suppose that $\textrm{gcd}(p_1,x_1)=\textrm{gcd}(p_1,x_2)=1$. Then as before, a choice for $\lambda$ is unsuitable if either $\lambda=(a_2/x_1)$ or $\lambda=-(a_1/x_2)$. Since $p_1$ is odd, there is a choice for $\lambda\in\Z_{p_1}$ such that both these choices are avoided, and so the claim holds. 

Let $n'=n/p_r$ and suppose (by induction on $r$) that there are $a_1, a_2$ such that $\textrm{gcd}(a_i,n')=1$ for $i=1,2$ such that $a_1x_1+a_2x_2=1$. Consider $a_1^*=a_1+\lambda n'x_2, a_2^*=a_2-\lambda n'x_1$. Since $\textrm{gcd}(n',p_r)=1$, arguing in the same manner as before, it follows that there is a choice for $\lambda$ such that $a_1^*, a_2^*$ are both non-zero in $\Z_{p_r}$. Furthermore,  $\textrm{gcd}(a_1^*,n')=\textrm{gcd}(a_2^*,n')=1$, so the proof of the lemma is complete for the case $m=2$ when $\textrm{gcd}(x_1,x_2)=1$.  In the general case, if $\textrm{gcd}(x_1, x_2)=d$ then consider $x_1'=x_1/d, x_2'=x_2/d$ so that $(x_1',x_2')=1$ and use the preceding argument with $x_1', x_2'$ instead to get $a_1,a_2$ as desired. 
  
 Suppose now that the lemma holds for smaller values than $m$. Write $d=\textrm{gcd}(x_1,\ldots,x_m)$ and let $\textrm{gcd}(x_1,\ldots,x_{m-1})=d'$ so that $\textrm{gcd}(d',x_m)=d$. By the case $m=2$, there exist $a, b$ such that $\textrm{gcd}(a,n)=\textrm{gcd}(b,n)=1$ and $ad'+bx_m=d$. Again by induction, there exist $a_i'$ for $1\le i\le m-1$ such that $\textrm{gcd}(a_i',n)=1$ and $\displaystyle\sum_{i=1}^{m-1}  a_i'x_i =d'$. Then it follows that $a_i=aa_i'$ for $1\le i\le m-1$ and $a_m=b$ satisfy the requirements of the statement of the lemma, and this completes the induction.
 \end{proof}
 
 Let $n=\prod_{i=1}^r p_i$ be a product of distinct odd primes and suppose $m=\omega(n)$. Let $\mfX_{2m}:=(X_1,\ldots,X_{2m})$ be a random $\Z_n$-sequence. We make the following claim:
 \begin{claim} \label{gcd} {\it whp} $gcd(X_1,\ldots,X_m)$ and $gcd(X_{m+1},\ldots, X_{2m})$ are relatively prime to $n$.\end{claim} 
 \begin{proof}(Proof of Claim \ref{gcd}) For any element $x\in \Z_n$,  let $\pi(x):=\{i: p_i\textrm{ divides } x\}$. Suppose $X$ is picked uniformly from $\Z_n$.
 For a fixed set $A\subset [r]$, we claim that $$\bP(\pi(X)=A)= \prod_{i\in A}\frac{1}{p_i}\prod_{i\notin A}\left(1-\frac{1}{p_i}\right).$$ In other words, each element $i\in [r]$ is picked independently into $A$ with probability $1/p_i$. This follows easily from the Chinese Remainder Theorem: $\Z_n\simeq \Z_{p_1}\times\cdots\times\Z_{p_r}$ as rings. If  $\pi_i:\Z_n\to\Z_{p_i}$ is the projection map then $p_i$ divides $x$ if and only if $\pi_i(x)=0$ and the claim follows immediately.
 
Consequently,  for each $i$, the probability that $p_i$ divides $X_j$ for all $1\le j\le m$ is $1/p_i^m$ so that
\begin{eqnarray*}\bP(\textrm{There exists } i\textrm{ such that } p_i\textrm{ divides } X_j\textrm{ for all } j)&\le&\sum_{i=1}^r\frac{1}{p_i^m}\\
  &\le& \frac{1}{p_1^{m-2}}\sum_{p\textrm{\ prime}} \frac{1}{p^2}\\ 
  &\le& \frac{1}{p_1^{m-2}}\left(\sum_{n\ge 1}\frac{1}{n^2}\right)=o(1) \end{eqnarray*} since $m\ge \omega(n)$.
Hence it follows that {\it whp} both $(X_1,\ldots, X_m)$ and $(X_{m+1},\ldots, X_{2m})$ are relatively prime to $n$. \end{proof}
We are now in a position to complete the proof of the third part of Theorem \ref{RandDav}. By lemma \ref{relprime} there exist $a_i,b_i\in\Z_n^*$ such that {\it whp} $a_1X_1+\cdots+a_mX_m=u_1$ and $b_1X_{m+1}+\cdots+b_mX_{2m}=u_2$, for $u_1,u_2\in\Z_n^*$.  Hence $$\sum_{i=1}^m (-u_2a_i) X_i+\sum_{i=1}^m (u_1b_i)X_{m+1}=0$$ holds {\it whp} and this establishes that {\it whp} $\mfX_{2m}$ is a Davenport Z-sequence and the proof is complete.

 \end{enumerate} \end{proof}
{\bf Remark:}   As mentioned in the introduction, and as is evident from the proof, the bound $m\ge\omega(n)$ in the last part of Theorem \ref{RandDav} is not always necessary. Indeed, for $\ep>0$, if  $p_1>\ep^{-1}$ then the same proof shows that for $m=3$ the probability that a random sequence $\mfX_3$ of length $3$ is a Davenport Z-sequence is at least $1-\ep$. 

\section{Concluding remarks}
\begin{itemize}
\item As we remarked, our proof of Theorem \ref{X_k(m)}, part 2, works only for $p$ prime, since we invoke the Cauchy-Davenport Theorem there. A more general result due to Kneser (see \cite{Nath}, chapter 4) provides a lower bound for all general $n$, but it is not clear how to adapt our proof effectively to the general case, or even to the case $n=p^m$ for $m\ge 2$, and $p$ prime. But we believe the following conjecture holds: 
\begin{conj}\label{conj4}  
 For any integer $n$ and any integer $k$, 
 $$\mathsf s^{(k)}(\Z_n) = n+k.$$
\end{conj}
\item Our proofs of Theorems \ref{RandErd} and \ref{RandDav} also give a bound on the error probability of their corresponding statements. A more interesting version of the same problem is the following  quantitative avatar: Suppose $\ep>0$, and suppose $A\subset\Z_n\setminus\{0\}$, and let $\mfX_m:=(X_1,\ldots,X_m)$ be a random $\Z_n$-sequence. Determine the maximum and minimum $m:=m_A(n,\ep), M:=M_A(n,\ep)$ respectively such that 
\begin{eqnarray*} \bP(\mfX_m\textrm{ is not a weighted Davenport Z-sequence for }A)&\ge& 1-\ep\\ \bP(\mfX_M\textrm{ is a weighted Davenport Z-sequence for }A)&\ge& 1-\ep\end{eqnarray*} hold. For instance, our proof of the first part of Theorem \ref{RandDav} establishes that for the set $A=\{1\}$ we have the bounds $m_A(n,\ep)\le \log_2 (\ep n)$ and $M_A(n,\ep)\ge\log_2(\frac{n}{\ep})$.  It would be interesting to see how much of an improvement is possible on these results and that would probably need more sophisticated probabilistic techniques. 
\item The proof of the last part of Theorem \ref{RandDav} establishes that sufficiently long sequences are already Davenport Z-sequences {\it whp}. We also remarked after the proof of the theorem that the same is not required; if the smallest prime factor on $n$ is sufficiently large, then bounded length sequences are already Davenport Z-sequences with high probability.  But whether $m=\omega(n)$ is {\it necessary in some cases} is not quite apparent.
\item The authors had previously proposed an extremal problem relating to the weighted Davenport constant of a group in \cite{BM1} which goes as follows. Given a finite abelian group $G$, and an integer $k\ge 2$,  we define $f^{(D)}_G(k)$ to be the least integer $\ell$ such that there is some subset $A\subset [1,\exp(G)-1]$ of size $\ell$ such that $D_A(G)\le k$. The most interesting case is for the group $G=\Z_p$ where $p$ is a prime and in this case, it turns out that \cite{BM2} in this case $$p^{1/k}<f^{(D)}_G(k)<4^{k^2}p^{1/k}.$$ As one of the principal motifs of this paper is that random $\Z_n$-sequences of much smaller length than $D_A(G)$ are sufficient to get weighted Davenport Z-sequences, it is natural to pose the following problem: Suppose $\ep>0$ and let $p$ be a prime, and let $k\ge 2$ be a fixed positive integer. Suppose $\mfX_k=(X_1,\ldots,X_k)$ is a random $\Z_p$-sequence. Let 
$$\A_{\ep}:=\{A\subset [1,p-1]: \bP(\mfX_k\textrm{ is an }A\textrm{-weighted Davenport Z-sequence })\ge 1-\ep\}.$$
Determine $$f^{(D)}_{\textrm{Rand}}(p,k,\ep):=\min_{A\in\A_{\ep}} |A|.$$
By the result in \cite{BM2} it follows that $f^{(D)}_{\textrm{Rand}}(p,k,\ep)<4^{k^2}p^{1/k}$. A preliminary guess based on the results here suggests that it is likely that $f^{(D)}_{\textrm{Rand}}(p,k,\ep)=o(p^{1/k})$.
\item Finally, two new problems suggest themselves rather naturally as a consequence of the result in the second part of Theorem \ref{RandDav}. Let $\ep>0$ and recall the notions of $M_A(n,\ep)$ and $\A_{\ep}$  as defined above.
\begin{enumerate}
\item[(a)] Determine $$\alpha(n,\ep):=\max\left\{\frac{M_A(n,\ep)}{D_A(\Z_n)}: A\in\A_{\ep}\right\}.$$  
The second part of Theorem \ref{RandDav} tells us that $\alpha(\ep, n)\ge\log_3 2-o(1)$. It would be interesting to see if this can be improved.
\item[(b)] Suppose $\ep>0$ and let $K\ge 1$ be a fixed parameter, and let $\mfX_m$ be a random $\Z_n$-sequence. Determine 
$$s(n,\ep, K):= \max\left\{|A|: A\in\A_{\ep}, M_A(n,\ep)\ge \frac{\alpha}{K}D_A(\Z_n)\right\}.$$ Again, the second part of Theorem \ref{RandDav} tells us that $s(n,\ep, K)\ge 2$ for all $K\ge 1$ and all $\ep>0$. Can it get significantly bigger? That is not clear at the moment.
\end{enumerate}
\end{itemize}

\section*{Acknowledgments} The authors are grateful to the anonymous referees for their careful reading and for pointing out discrepancies and errors that were there in the original manuscript and also for pointing out to some inaccuracies in the proofs which led us to rewrite significant portions of the paper, which has improved the quality of the paper overall.

 \section{Appendix I: Szemer\'edi's proof of the Erd\H{o}s-Eggleston Conjecture}
 
In this section, we present a proof of the Erd\H{o}s-Eggleston conjecture that was settled by Szemer\'edi. In fact, Szemer\'edi's proof works for any abelian group.  
Our presentation of the proof alone is ours, as these ideas are 
all there in Szemer\'edi's paper. We make no claim regarding the 
optimality of the constant that appears here,  nor do we make 
any attempts to optimize. We shall also drop ceilings and floors to 
make the presentation simpler. Recall that for a set $A\subset G$, by $\Sym(A)$ we mean the set of all sums $\sum_{x\in X} x$ as $X$ varies over all non-empty subsets of $A$.
\begin{thm} {(\bf Szemer\'edi}) Let $G$ be a finite abelian group and suppose $A\subset G$ such that $0\not\in\Sym(A)$. Then $|\Sym(A)|\ge\frac{|A|^2}{10000}$.\end{thm}
\begin{proof} Suppose the statement of the theorem does not hold; in particular suppose $A\subset G$ satisfies $0\not\in \Sym(A)$ and $|\Sym(A)|<\frac{|A|^2}{10000}$. 
We may assume without loss of generality that $|A|$ is sufficiently large (how large $|A|$ could be assumed to be will be determined shortly by the inequalities that shall appear). Let us write $|A|=\ell$. For each $\ell/4\le k\le 3\ell/4$ we define the bipartite graph $G_k$ with vertex sets $\binom{A}{k}, \binom{A}{k+1}$ (For a set $A$, $\binom{A}{r}$ denotes the set of all $r$-subsets of $A$) as follows.  $X\in\binom{A}{k}$ is adjacent to $Y\in \binom{A}{k+1}$ in $G_k$ if and only if  $X\subset Y$ and $|\Sym(Y)\setminus\Sym(X)|\le \ell/100$. The upshot of this definition for the graphs $G_k$ is this: If we consider the union of all the graphs $\G=\cup_k G_k$, and consider any chain $X_0\subset\cdots \subset X_{\ell/2}$ of sets with $|X_i|=\ell/4+i$ (for $0\le i\le \ell/2$), then there are at most $\ell/100$ `missing edges' along the chain in the graph $\G$. This follows since if there are more than $\ell/100$ missing edges along some chain $(X_1,\ldots,X_{\ell/2})$ then $|\cup_i \Sym(X_{i+1})\setminus \Sym(X_i)|> (\ell/100)(\ell/100)$ and that contradicts the assumption that the statement is false.

Fix integers $t, k$ and consider $D\in\binom{A}{k}$. Denote by $\textrm{deg}_k(D)$ the degree of $D$ in $G_k$. Suppose $\textrm{deg}_k(D), \textrm{deg}_{k-1}(D)\ge t$. Let $D$ have neighbors $B_1,\ldots, B_t$ and $A_1,\ldots, A_t$ in $G_k$ and $G_{k-1}$ respectively. Write  $A_i=D\setminus\{a_i\}$, and $B_i=D\cup\{b_i\}$. Since 
$$S:=\left\{\sum_{x\in D} x -a_i+b_j:i,j\in[t]\right\}\supset\Sym(B_j)$$ 
and $|\Sym(B_j)\setminus\Sym(D)|\le\ell/100$, there are at least $t-\ell/100$ elements in $S\cap \Sym(D)$.
Since this holds for each $j$, a simple averaging argument implies that there exists $i$ such that $t-\ell/100$ elements of the set $U:=\{\sum_{x\in D} x-a_i+b_j:1\le j\le t\}$ belong to $\Sym(D)$, and since $|\Sym(D)\setminus\Sym(A_i)|\le\ell/100$, it follows again that at least $|U\cap \Sym(A_i)|\ge t-\ell/50$. In particular,  if $t>\ell/50$, there is some $1\le j_0\le t$ and $D_1\subset D\setminus\{a_i\}$ such that
$$\sum_{x\in D} x - a_i+b_{j_0}=\sum_{y\in D_1} y, \text{ \  which implies that } b_{j_0}+\sum_{x\in D\setminus (D_1\cup\{a_i\})} x=0$$ contradicting the hypothesis that $0\not\in\Sym(A)$.

So, to complete the proof, we need to show that there exists $k\in [\ell/4+1,3\ell/4]$ and $D\in\binom{A}{k}$ such that $D$ has degree more than $\ell/50$ in both $G_k$ and $G_{k-1}$. Towards that end, let us denote by $d(G_k)$ the \textit{density} of  $G_k$, i.e., $d(G_k)=\frac{e(G_k)}{m_k}$ where $e(G_k)$ denotes the number of edges in $G_k$ and $m_k$ denotes the cardinality of the set $\{(X,Y):X\in\binom{A}{k}, Y\in\binom{A}{k+1}, X\subset Y\}$. 

We now claim the following: If there exists $k$ in this range such that both $d(G_k),d(G_{k-1})\ge 2/3$, then there is a vertex $D\in\binom{A}{k}$ such that $\textrm{deg}_k(D),\textrm{deg}_{k-1}(D)>\ell/50$. By the preceding observation, the proof of the theorem follows as a consequence. To prove the claim, suppose $\textrm{BAD}_k:=\{D\in\binom{A}{k}: \textrm{deg}_k(D)\le\ell/48\}$. Then 
$$ \frac{2}{3}\binom{\ell}{k}(\ell-k)=\frac{2m_k}{3}\le e(G_k)\le |\textrm{BAD}_k|(\ell/48)+\left(\binom{\ell}{k}-|\textrm{BAD}_k|\right)(\ell-k)$$
which gives $|\textrm{BAD}_k|\le\frac{18}{44}\binom{\ell}{k}<\frac{1}{2}\binom{\ell}{k}$. Similarly, we get $|\textrm{BAD}_{k-1}|<\frac{1}{2}\binom{\ell}{k}$, so there exists $D\in\binom{A}{k}\setminus(\textrm{BAD}_k\cup \textrm{BAD}_{k-1})$ and for this $D$, we have $\textrm{deg}_k(D)\ge \ell/48, \textrm{deg}_{k-1}(D)\ge\ell/48$, and that achieves our goal.

 So finally, to establish that for some $k$ we have $d(G_k),d(G_{k-1})\ge 2/3$, 
 we revert to our  observation at the very beginning of the proof: Suppose $X_i\in\binom{A}{\ell/4+i}$ for $0\le i\le\ell/2$ such that $X_i\subset X_{i+1}$ for each $i$, and let $\C$ denote the chain of subsets $\C=(X_0,\ldots,X_{\ell/2})$ with $X_i\subset X_{i+1}$ for each $0\le i\le \ell/2-1$. Then among the pairs $(X_i,X_{i+1})$ at most $\ell/100$ are non-edges in $\G$. Equivalently, for every chain $\C=(X_0,\ldots,X_{\ell/2})$, if we uniformly and randomly pick $i$ with $0\le i\le\ell/2-1$, then the probability that $(X_i,X_{i+1})$ is an edge of $\G$ is at least $0.99$.
 
 Call $k$ a {\it Bad level}, if $d(G_k)<2/3$. If what we seek does not hold, then every alternate level is Bad. In particular, if $\C=(X_0,\ldots,X_{\ell/2})$ is a chain chosen uniformly and randomly and $\textrm{Miss}(\C)$ denotes the number of missing edges along the sets in $\C$, then
 $$\ell/100\ge\bE(\textrm{Miss}(\C))\ge\sum_{k\textrm{\ Bad}} \bP((X_i,X_{i+1})\textrm{\ is\ not\ an\ edge})\ge \frac{1}{2}\cdot\frac{\ell}{2} $$
(the first inequality follows by assumption on $\textrm{Miss}(\C)$ for all chains $\C$, and the last inequality follows since there are at most $\ell/2$ levels, and at least half of those are Bad by assumption)
and that is a contradiction.\end{proof}
\section{Appendix II: Proof of Claim \ref{Uncorr}}
Suppose $\{I,J\}\ne \{I',J'\}$ satisfy $I\cap J=I'\cap J'=\emptyset$. As in the previous case of Theorem \ref{RandDav} we compute $\bP(\mfX_I=\mfX_J\textrm{ and } \mfX_{I'}=\mfX_{J'})$, and use Lemma \ref{pwi}. Towards that end, we shall consider the system of equations (\ref{uniquesoln}) for a specific choice for the special indices $i_0,j_0$ and show that except in the exceptional cases mentioned in Claim \ref{Uncorr}, the system of equations admit unique solutions for $(X_{i_0}, X_{j_0})$ which implies that $\bP(\mfX_I=\mfX_J\textrm{ and } \mfX_{I'}=\mfX_{J'})=1/n^2$ which in turn implies that $\Cov(\mathbb{I}(\mfX_{I,J}),\mathbb{I}(\mfX_{I',J'}))=0$ Even in these exceptional cases too, we shall show that there are at most two solutions for the pair $(X_{i_0}, X_{j_0})$ (like in the case of the proof of  Theorem \ref{RandErd}, part 3) which implies that in the bad cases, $\Cov(\mathbb{I}(\mfX_{I,J}),\mathbb{I}(\mfX_{I',J'}))\le 2/n^2$ as desired.

Suppose first that $I=\emptyset$. If $I'=\emptyset$ too, then both $J,J'\ne\emptyset$. If both $J\setminus J'$ and $J'\setminus J$ are nonempty, then pick $i\in J\setminus J', j\in J'\setminus J$. Let $T=I\cup J\cup I'\cup J'\setminus\{i,j\}$ and let $\bfx$ be an arbitrary $\Z_n$-sequence indexed by the elements of $T$. Consider the equations in (\ref{uniquesoln}); this reduces to two  equations of the form $X_i=\xi, X_j=\xi'$, which admits a unique solution for $(X_i,X_j)$ in $\Z_n$. If $J\subset J'$, then again, pick $i\in J$ and $j\in J'\setminus J$. This sets up equations of the form $X_i=\xi, X_i+X_j=\xi'$ which again admit a unique solution for $(X_i,X_j)$. In fact, the same argument also works if $I=I'$ as well. 

Next suppose $I=\emptyset$ but $J,I',J'\neq\emptyset$.  In this case, first suppose that $J\setminus(I'\cup J')\ne\emptyset$. In this case, pick $i\in J\setminus(I'\cup J')$ and $j\in J'$; this gives equations $X_i=\xi,X_j=\xi'$ if $j\in  J'\setminus J$ and equations $X_j=\xi, X_i+X_j=\xi'$ if $j\in J\cap J'$. In either case, this admits a unique solution for the pair $(X_i,X_j)$. If $J\subsetneq(I'\cup J')$,  pick $i\in(I'\cup J')\setminus J$ and $j\in J$ to get equations of the form $X_j=\xi,X_i\pm X_j=\xi'$ (depending on where $i$ and $j$ lie) which again leads to a unique solution for $(X_i,X_j)$. This finally leads us to the case $J=I'\cup J'$. This is the first of the exceptional cases in Claim \ref{Uncorr}.

Henceforth we shall assume  that all $I,J, I',J'\ne\emptyset$ and further that all the sets $I, I',J, J'$ are pairwise distinct.

Pick $i\in I$ and $j\in I'$. We now consider the various possibilities for the membership of elements $i,j$ in the sets $J, J'$ respectively. We shall merely write down the nature of the equations it imposes upon $X_i,X_j$. To keep our notation succinct, we shall denote the profiles by $0$-$1$ tuples $(a,b,c,d)$ which shall denote the following: $a=\bI_{i\in I'}, b=\bI_{i\in J'}, c=\bI_{j\in I}, d=\bI_{j\in J}$. So for instance, $(a,b,c,d)=(0,0,0,0)$ simply means $i\in \overline{I'\cup J'}$ and $j\in\overline{I\cup J}$, and so on. Observe that if $(a,b)$ and $(c,d)$ are interchanged, the result is merely the interchange of $i$ and $j$, so the nature of the equations is the same, so we shall club those instances into the same case. We shall merely note down the equations this forces upon $(X_i,X_j)$ (or $(X_j,X_i)$). Also note that since $I\cap J=I'\cap J'$ we cannot have $(a,b),(c,d)=(1,1)$; indeed, if $(a,b)=(1,1)$ then it implies $i\in I'\cap J'$.
\begin{enumerate}
\item $(a,b,c,d)=(0,0,0,0)$:  $X_i=\xi,X_j=\xi'$.
\item  $(a,b,c,d)=(0,0,0,1), (0,1,0,0)$: $X_i-X_j=\xi,X_j=\xi'$.
\item  $(a,b,c,d)=(0,0,1,0), (1,0,0,0)$:  $X_i+X_j=\xi,X_j=\xi'$.
\item  $(a,b,c,d)=(1,0,0,1)$: If this is the profile for every choice of $i\in I, j\in I'$ then it forces $I\subset I'\subset J$. But since $I\cap J=\emptyset$, this forces $I=\emptyset$ which has already been considered before.
\item $(a,b,c,d)=(1,0,1,0)$: In this case, this forces $I\subset I'\subset I$ which gives $I=I'$ and again, this case has already been dealt with earlier.
\item  $(a,b,c,d)=(0,1,0,1)$. In this final profile, the aforementioned argument then gives us that $I\subset J'$ and $I'\subset J$. If $J'\setminus (I\cup J)\ne\emptyset$, pick $i\in I, j\in J'\setminus(I\cup J)$ to get equations of the form $X_i=\xi, X_i+X_j=\xi'$ which admit unique solutions. Hence we may assume that $I\subset J'\subset I\cup J$. Here, pick $i\in I, j\in J'\setminus I$ and condition on all the other $X_{\ell}$. This gives equations $X_i-X_j=\xi, X_i+X_j=\xi'$ and in $\Z_n$, this system of equations admits a unique solution for $(X_i,X_j)$ if $n$ is odd, or admits at most $2$ solutions if $n$ is even. This completes the case-by-case analysis and the proof. 
\end{enumerate}


\begin{thebibliography}{AAAA}
\bibliographystyle{alpha} 
\bibitem{Adhetal}  S. D. Adhikari,  Y. G. Chen, J. B. Friedlander, S. V. Konyagin, and F. Pappalardi,  Contributions to zero-sum problems. \textit{Discrete Math.} {\bf 306} (2006), no. 1, 1-10.
\bibitem{AC} S. D. Adhikari, and Y. G. Chen, Davenport constant with weights and some related question II, \textit{J. Combin. Theory Ser. A} {\bf{115}} (2008), No. 1, 178-184.
\bibitem{BM1} N. Balachandran, E. Mazumdar, The Weighted Davenport constant of a group and a related extremal problem, \textit{Elect. J. Combin.}, {\bf 26} (2019), Issue 4, P4, 51.
\bibitem{BM2} N. Balachandran, E. Mazumdar, The Weighted Davenport constant of a group and a related extremal problem - II, \textit{https://arxiv.org/abs/1912.07509}.
\bibitem{EGZ} P. Erd\H{o}s, A. Ginzburg, and A . Ziv, Theorem in the additive number theory,  \textit{Bull. Res. Counc. Israel Sect. F Math. Phys.} {\bf 10F} (1961), no. 1, 41-43.
\bibitem{Gri} S. Griffiths, The Erd\H{o}s- Ginzburg - Ziv Theorem with units, \textit{ Discrete Math.} {\bf 308} (2008), no. 23, 5473 - 5484.
\bibitem{GGS} W. Gao, A. Geroldinger, and W. Schmid, Inverse zero-sum problems, \textit{Acta Arith.} 245-279.
\bibitem{Luca} F. Luca, A generalization of a classical zero-sum problem, \textit{Discrete Math.} {\bf 307} (2007), 1672-1678.
\bibitem{MORS} L. E. Marchan, O. Ordaz, D. Ramos and W. A. Schmid, Some Exact Values of the Harborth Constant and Its Plus-Minus Weighted Analogue, {\it Archiv der Mathematik } {\bf 101} (2013), 501-512 . 
\bibitem{Nath} M. B. Nathanson, {\it Additive Number Theory: Inverse Problems and the Geometry of Sumsets}, Graduate texts in Mathematics, Vol. 165, Springer-Verlag, New York, 1996.
\bibitem{Scherk} L. Moser and P. Scherk, Solution to advanced problem 4466, \textit{Amer. Math. Monthly},  {\bf 1} (1955), Vol. 62, 46-47.
\bibitem{Sze} E. Szemer\'edi, On a conjecture of Erd\H{o}s and Heilbronn, \textit{Acta Arith.}, {\bf 17} (1970), 227-229.
 \bibitem{Yu}H.B. Yu, A Simple Proof of a theorem of Bollob\'as and Leader.\textit{ Proc. American Math. Society.} {\bf 131} (2003), 9,  2639- 2640.
 \bibitem{YuanZeng} P. Yuan and X. Zeng, Davenport constant with weights.
 \textit{European J. Comb.} {\bf 31} (2010), 677-680.
 \end{thebibliography}
\end{document}